\documentclass[a4paper,twoside,11pt]{article}

\usepackage[T1]{fontenc}
\usepackage[ansinew]{inputenc}
\usepackage{geometry}
\usepackage{color} 
\usepackage{lmodern}
\usepackage{graphicx}
\usepackage{float}
\usepackage{amsmath} \numberwithin{equation}{section}
\usepackage{amsthm}
\usepackage{amsfonts}
\usepackage{amssymb}
\usepackage{graphics}
\usepackage[pdfstartview=FitH]{hyperref}
\usepackage{enumerate}
\usepackage{dsfont}

\theoremstyle{plain}
\newtheorem{theorem}{Theorem}[section]
\newtheorem{corollary}[theorem]{Corollary}
\newtheorem{lemma}[theorem]{Lemma}

\newtheorem{proposition}[theorem]{Proposition}
\newtheorem{remark}[theorem]{Remark}  
\newtheorem{definition}[theorem]{Definition}
\newtheorem{question}[theorem]{Question}

\newcommand\C{\mathbb C}      
\newcommand\R{\mathbb R}

\newcommand\N{\mathbb N}        
       
\newcommand\D{\mathbb D}

\renewcommand\Re{\operatorname{Re}}       
\newcommand{\eps}{\varepsilon}

\newcommand{\ex}{\operatorname{ex}} 
\newcommand{\supp}{\operatorname{supp}} 
\newcommand{\aut}{\operatorname{Aut}} 

\newcommand{\B}{\mathbb B}
 \newcommand{\M}{\mathcal{M}}
 \newcommand{\Ho}{\mathcal{H}} 
\newcommand\with{\ \vrule\ }  % 'mit'-Symbol in Mengen

\usepackage{tocloft}
\usepackage{hyperref}

\newcommand*\origfootnoterule{}
\let\origfootnoterule\footnoterule
\renewcommand*{\footnoterule}{\vskip 0pt plus .01fil\origfootnoterule}

\begin{document}
\parindent 0pt 

\title{On the parametric representation of univalent functions on the polydisc}
\author{Sebastian Schlei{\ss}inger}
\setcounter{section}{0}

\maketitle

\begin{abstract}
We consider support points of the class $S^0(\D^n)$ of normalized univalent mappings on the polydisc 
$\D^n$ with parametric representation and we prove sharp estimates for coefficients of degree 2. 
\end{abstract}


\section{Introduction}

Let $\D=\{z\in\C\,|\, |z|<1\}$ be the unit disc and let $f:\D\to \C$ be a univalent mapping normalized with 
$f(z)=z+\sum_{n\geq 2}a_nz^n$. The Bieberbach conjecture (see \cite{Bie}) states that 
$$ |a_n|\leq n \quad \text{for all $n\geq 2.$}$$
Loewner has proven the case $n=3$ in \cite{Loewner:1923} by introducing a new tool for the study of univalent functions, a 
``parametric representation'' for $f$ via a certain differential equation. 
The Bieberbach conjecture has been proven completely by de Branges; see \cite{dB85}. Univalent functions and Loewner theory have also been studied in higher dimensions.  A general theory for certain complex manifolds has been established in \cite{MR2507634}. \\

For $n\geq 2$, the most studied subdomains of $\C^n$ are the polydisc $\D^n$ and the Euclidean unit ball $\B^n.$ 
Also Loewner theory can be studied on these domains and, 
in particular, one can define normalized univalent functions having a parametric representation. These functions form a compact set and naturally lead to extremal problems,
e.g. finding coefficient bounds.\\
 While there is a lot of recent research on  extremal problems for functions with parametric representation on $\B^n$ 
 (\cite{MR2943779, GHKK14, bsp, pont, GHKK16, HIK16, BGHK16, braro, GHKK17}), the case of the polydisc gained only little 
 interest since Poreda's introduction of the class $S^0(\D^n)$ in 1987 (\cite{MR1049182, MR1049183}).\\
 
In what follows, we recapitulate the definition and some basic properties of $S^0(\D^n)$ in Section \ref{sec_2}. In Section \ref{sec_3}, we prove 
some statements for general support points of $S^0(\D^n)$ and in Section \ref{sec_4}, we prove estimates for all coefficients of degree 2
and give several examples showing that these estimates are sharp.\\

 \textbf{Update:} When writing this paper, the author was not aware of the 
 work \cite{GHK}, which proves several results 
 concerning parametric representation of univalent functions on $\D^n$. In particular, 
 the estimates from Theorem \ref{bie2} are shown there for an even more general setting.

\section{The classes \texorpdfstring{$\M(\D^n)$}{} and \texorpdfstring{$S^0(\D^n)$}{}}\label{sec_2}

We denote by $\Ho(\D^n,\C^n)$ the set of all holomorphic mappings $f:\D^n\to \C^n.$\\
Furthermore, we let $S(\D^n)$ be the set of all univalent functions $f\in \Ho(\D^n,\C^n)$ with $f(0)=0$ and $Df(0)=I_n.$ This class is not compact 
when $n\geq 2$, as the simple examples $$(z_1,z_2)\mapsto (z_1+nz_2^2, z_2), \quad n\in\N,$$ show.
In \cite{MR1049182}, Poreda introduced the class $S^0(\D^n)$ as the set of all $f\in S(\D^n)$ having 
``parametric representation''. Later,  G. Kohr defined the corresponding class $S^0(\mathcal{\B}^n)$ for the unit ball, see \cite{MR1929522}, which has been extensively studied since its introduction.\\
The class $S^0(\D^n)$ is defined via Loewner's differential equation. First, one considers the set
$$ \mathcal{M}(\D^n):=\left\{h\in\Ho(\D^n,\C^n)\;|\; h(0)=0, Dh(0)=-I_n, \Re\left(\frac{h_j(z)}{z_j}\right)\leq 0 \;\text{when}\;
\|z\|_\infty=|z_j|>0 \right\}. $$

\begin{remark}\label{dim_one}
For $n=1,$ we have $$\M(\D)=\{z\mapsto -zp(z)\,|\, p\in \mathcal{P}\},$$
where $\mathcal{P}$ denotes the Carath\'{e}odory class of all holomorphic functions
$p:\D\to \C$ with $\Re(p(z))>0$ for all  $z\in\D$ and $p(0)=1$.
The class $\mathcal{P}$ can be characterized by the Riesz-Herglotz representation formula:
\begin{equation*}   \mathcal{P} = \left\{
\int_{\partial \D} \frac{u + z}{u -z} \, \mu(du) \,|\, \mu \; 
\text{is a probability measure on $\partial \D$} \right\}.  \end{equation*}
A simple consequence is the following coefficient bound (Carath\'{e}odory's lemma):\\ Write $p(z)=1+\sum_{n\geq 1} c_n z^n.$ 
Then 
\begin{equation}\label{Chris}   |c_n|\leq 2 \quad \text{for all $n\geq 1.$}  \end{equation}
Equality holds, e.g., if $p(z)=\frac{u+z}{u-z}$ for some $u\in \partial \D,$
i.e. when $\mu$ is a point measure; see \cite[Corollary 2.3]{Pom:1975} for a complete characterization.
\end{remark}

The class $\M(\D^n)$ is closely related to the class $S^*(\D^n)$ of all starlike functions, i.e. those $f\in S(\D^n)$ such that $f(\D^n)$ is a starlike domain with respect to $0$.

\begin{theorem}[\cite{MR0261040}]\label{Juergen}
Let $f:\D^n \to \C^n$ be locally biholomorphic, i.e. $Df(z)$ is invertible for every $z\in\D^n,$ with $f(0)=0$ and $Df(0)=I_n$. Then $f\in S^*(\D^n)$  if and only if the function 
$z\mapsto -(Df(z))^{-1}\cdot f(z)$ belongs to $\M(\D^n).$
\end{theorem}

Loosely speaking, the class $S^0(\D^n)$ can be thought of as all mappings that can be written as an infinite composition of infinitesimal starlike mappings on 
$\D^n.$ This idea is made precise by using a differential equation involving the class $\M(\D^n).$\\

We define a \emph{Herglotz vector field} $G$ as a mapping $G:\D^n\times [0,\infty) \to \C^n$ with $G(\cdot,t)\in \M(\D^n)$ for all $t\geq 0$ such that $G(z,\cdot)$ is measurable on $[0,\infty)$ for all $z\in\D^n$. The corresponding Loewner equation is given by 
\begin{equation}\label{Loewner_ODE}
 \frac{\partial{\varphi_{s,t}(z)}}{\partial t}=G(\varphi_{s,t}(z),t) \quad \text{for almost all}\quad t\geq s, \quad \varphi_{s,s}(z)=z.
\end{equation}

The solution $t\mapsto \varphi_{s,t}$ is a family of univalent functions $\varphi_{s,t}:\D^n\to \D^n$ normalized by $\varphi_{s,t}(0)=0,$ $D\varphi_{s,t}(0)=e^{s-t}I_n$. 
The family $\{\varphi_{s,t}\}_{0\leq s \leq t}$ satisfies the algebraic property 
\begin{equation}\label{algebra}
\text{$\varphi_{s,t}=\varphi_{u,t}\circ\varphi_{s,u}$ for all $0\leq s\leq u\leq t$}
\end{equation} and is called an 
\emph{evolution family}.\\
This notion is closely related to Loewner chains. We define a \emph{normalized Loewner chain} on $\D^n$ as a family $\{f_t\}_{t\geq 0}$ of univalent mappings
 $f_t:\D^n\to \C^n$ with $f_t(0)=0, Df_t(0)=e^tI_n$ and $f_s(\D^n)\subseteq f_t(\D^n)$ for all $0\leq s\leq t.$

One can construct normalized Loewner chains from \eqref{Loewner_ODE} as follows. 

\begin{theorem}[\cite{MR1049182}]\label{Gabriela}
 Let $G(z,t)$ be a Herglotz vector field. For every $s\geq 0$ and $z\in \D^n,$ let $\varphi_{s,t}(z)$ be the solution of the initial value problem \eqref{Loewner_ODE}.
Then the limit \begin{equation}\label{represent}\lim_{t\to\infty}e^t\varphi_{s,t}(z)=:f_s(z)\end{equation}
exists for all $s\geq 0$ locally uniformly on $\D^n$ and $f_s \in S(\D^n)$.\\
Furthermore, the functions $\{f_t\}_{t\geq 0}$ satisfy $f_s(z)=f_t(\varphi_{s,t}(z))$ for all $z\in \D^n$ and $0\leq s\leq t,$ and $\{f_t\}_{t\geq 0}$ 
is a normalized Loewner chain having the property that $\{e^{-t}f_t\}_{t\geq0}$ is a normal family on $\D^n.$ Finally, $f_t$ satisfies the Loewner PDE
$$\frac{\partial{f_t(z)}}{\partial t}=-Df_t(z)G(z,t) \quad \text{for all}\quad z\in \D^n\quad \text{and for almost all }\quad t\geq 0.$$
\end{theorem} 

The first element $f_0\in S(\D^n)$ of the Loewner chain in the theorem above is said to have \textit{parametric representation.}
\begin{definition}
 $S^0(\D^n):=\{f\in S(\D^n)\with \, f \; \text{has parametric representation}\}.$
\end{definition}

\begin{proposition}\label{daniel} Let $G$, $\varphi_{s,t}$ and $f_t$ be defined as in Theorem \ref{Gabriela}. 
 \begin{itemize}
  \item[a)] For all $0\leq s\leq t,$ $e^{-s}f_s \in S^0(\D^n)$ and $e^{t-s} \varphi_{s,t} \in S^0(\D^n).$
  \item[b)] $f\in S^0(\D^n)$ if and only if there exists a normalized Loewner chain $\{f_t\}_{t\geq0}$ with $f=f_0$ such that $\{e^{-t}f_t\}_{t\geq 0}$ is a normal family on $\D^n.$
\item[c)] $S^*(\D^n)\subset S^0(\D^n).$ 
 \end{itemize}
\end{proposition}
\begin{proof}${}$
  \begin{itemize}
  \item[a)] Define $H(z,\tau):=G(z,\tau+s)$ for $\tau\in [0,t-s]$ and $H(z,\tau)=-z$ for $\tau > t-s.$ Denote the solution of \eqref{Loewner_ODE} for the Herglotz vector field 
$H$ by $\psi_{s,t}.$ Then $\psi(0,\tau) = \psi(t-s,\tau) \circ \psi(0,t-s) = e^{-\tau+t-s} \cdot \psi_{0,t-s}=e^{-\tau+t-s} \cdot \varphi_{s,t}$ for $\tau > t-s$. 
Hence $e^{\tau}\psi(0,\tau)=e^{t-s} \cdot \varphi_{s,t} \to e^{t-s} \cdot \varphi_{s,t}$ as $\tau \to \infty.$ \\
Similarly, the mapping $e^{-s}f_s$ can be generated by the Herglotz vector field $H(z,t)=G(z,t+s).$
  \item[b)] See \cite[Corollary 2.5]{GKK03}.
  \item[c)] If $f\in S^*(\D^n)$, then $\{e^t f\}_{t\geq 0}$ is a normalized Loewner chain and we conclude from b) that $f \in S^0(\D^n).$ 
  The corresponding Herglotz vector field is constant w.r.t. time, i.e. $G(z,t) = -(Df(z))^{-1}\cdot f(z).$ 
 \end{itemize}

\end{proof}

Elements of the class $S^0(\D^n)$ enjoy the following inequalities, which are known as the Koebe distortion theorem when $n=1.$

\begin{theorem}[Theorem 1 and Theorem 2 in \cite{MR1049182}]\label{koebe2}
 If $f\in S^0(\D^n)$, then $$\frac{\|z\|_\infty}{(1+\|z\|_\infty)^2} \leq \|f(z)\|_\infty\leq \frac{\|z\|_\infty}{(1-\|z\|_\infty)^2} \quad \text{for all} \quad z\in \D^n. $$
In particular, $\frac1{4}\D^n \subseteq f(\D^n).$ (Koebe quarter theorem for the class $S^0(\D^n).$)
\end{theorem}

This can be used to prove:

\begin{theorem}[Theorem 2.9 in \cite{GKK03}]
 The class $S^0(\D^n)$ is compact.
\end{theorem}

\begin{remark} In one dimension, we have $S^0(\D)=S(\D)$ (\cite[Theorem 6.1]{Pom:1975}), which cannot be true in higher dimensions as 
$S(\D^n)$ is not compact for $n\geq 2$. \\
 There is a somehow geometric property for domains related to $S^0(\D^n),$ called \emph{asymptotic starlikeness}. This notion was introduced by Poreda in \cite{MR1049183}. He showed that this property is a necessary condition for a domain to be the image of a function $f\in S^0(\D^n).$ Under some further assumptions this condition is also sufficient. In \cite[Theorem 3.1]{MR2425737}, it is shown that $f:\B_n\to\C^n$ has parametric representation on the unit ball if and only if $f$ is univalent, normalized, and $f(\B_n)$ is an asymptotically starlike domain. 
\end{remark}

We summarize some further properties of the class $S^0(\D^n)$. Property $b)$ 
will be essential for the proof of Theorem \ref{sup}.

\begin{theorem}\label{range}
Let $f\in S^0(\D^n)$ and let $\{f_t\}_{t\geq0}$ be a normalized Loewner chain with $f=f_0$ such that $\{e^{-t}f_t\}_{t\geq0}$ is a normal family. Then
\begin{itemize}
 \item[a)] $\bigcup_{t\geq0}f_t(\D^n)=\C^n$. 
\item[b)] $f(\D^n)$ is a Runge domain.
\item[c)] For $n\geq 2$, $S^*(\D^n)\cap \operatorname{Aut}(\C^n)$ is dense in $S^*(\D^n)$ and $S^0(\D^n)\cap \operatorname{Aut}(\C^n)$ is dense in $S^0(\D^n)$.
\end{itemize}
\end{theorem}
Here, we don't distinguish between $f\in \aut(\C^n)$ and its restriction $f|_{\D^n}$ to $\D^n$ to simplify notation.

\begin{proof}
a) Proposition \ref{daniel} a) and Theorem  \ref{koebe2} imply $$\bigcup_{t\geq 0}f_t(\D^n)\supseteq \bigcup_{t\geq 0}\left(\frac{e^t}{4}\cdot \D^n\right)=\C^n.$$ 
b) Consequently, the Loewner chain $\{f_t\}_{t\geq 0}$ extends $f(\D^n)$ to the Runge domain $\C^n.$ This is a special case of the ``semicontinuous 
holomorphic extendability'' (to $\C^n$) defined  in \cite{MR0148939} by Docquier and Grauert. They proved that this impies that $f(\D^n)$ is a Runge domain; 
see \cite[Satz 19]{MR0148939}. We also refer to \cite[Theorem 4.2]{ABWold} for an English reference.\\

c) We start with the case $f\in S^*(\D^n).$ As $f$ maps $\D^n$ onto a Runge domain, it can be approximated locally uniformly on $\D^n$ by a sequence  
$(g_{k})_k\subset  \aut(\C^n),$ see Theorem 2.1 in \cite{MR1185588}. We may assume that $g_k(0)=0$ and $Dg_k(0)=I_n$. \\
Now we also have $f_r:=\frac1{r}f(rz)\in S^*(\D^n)$ for every $r\in (0,1)$ and $g_{k,r}:=\frac1{r}g_{k}(rz)$ converges uniformly on $\D^n$ to $f_r$ as $k\to\infty.$ 
We have $-(Df_{r})^{-1}\cdot f_{r}\in \M(\D^n)$, and thus $-(Dg_{k,r})^{-1}\cdot g_{k,r}\in \M(\D^n)$ 
 for all $k$ large enough, say $k\geq K_r.$ Hence $g_{k,r}\in S^*(\D^n)$ for all $k\geq K_r$.\\
Consequently, the sequence $(g_{K_{r_m},r_m})_{m}$, with $r_m=1-1/m$, belongs to $S^*(\D^n)\cap \operatorname{Aut}(\C^n)$ and converges locally uniformly on $\D^n$ to $f.$\\  

Next let $f$ be an arbitrary mapping from $S^0(\D^n).$ Then $f= \lim_{t\to\infty}e^t \varphi_{0,t}$ where $\varphi_{0,t}$ is a solution to 
\eqref{Loewner_ODE} with a Herglotz vector field $G$. So it suffices to approximate $e^T\varphi_{0,T}$ for every 
$T>0$ by automorphisms of $\aut(\C^n)$ that belong to 
$S^0(\D^n)$.\\
First, we approximate $G$ by a sequence of piecewise constant Herglotz vector fields $G_k$ such that the corresponding solution 
$\varphi^k_{0,T}$ of \eqref{Loewner_ODE} for $G_k$ at time $t=T>0$ converges locally uniformly on $\D^n$ to $\varphi_{0,T}$ as $k\to\infty$. \\
We can further assume that every constant has the form 
$-(Dg)^{-1}\cdot g$ for some $g\in \aut(\C^n)\cap S^*(\D^n).$ Due to property \eqref{algebra}, the mapping $\varphi^k_{0,T}$ is a composition of 
automorphisms of $\C^n$, so $\varphi^k_{0,T}\in \aut(\C^n).$ With Proposition \ref{daniel} a), we conclude that $e^T \varphi^k_{0,T}\in S^0(\D^n)\cap \aut(\C^n)$.
\end{proof}

\section{Extreme and support points of \texorpdfstring{$S^0(\D^n)$}{}} \label{sec_3}

Let $X$ be a locally convex $\C$--vector space and $E\subset X.$ The set $\ex E$ of extreme points  and the set $\supp E$ of support points  of $E$ are defined as follows:
\begin{itemize}
\item[\textbullet] $x\in \ex E$ if the representation $x=ta + (1-t)b$ with $t\in[0,1],$ $a,b\in E$,  always implies $x=a=b.$
\item[\textbullet] $x\in \supp E$ if there exists a continuous linear functional $L:X\to\C$ such that $\Re L$ is non-constant on $E$ and $$\Re L(x) = \max_{y\in E}\Re L(y).$$
\end{itemize}

The class $S^0(\D^n)$ is a nonempty compact subset of the locally convex vector space $\Ho(\D^n, \C^n).$ Thus the Krein--Milman theorem 
implies that $\ex S^0(\D^n)$ is nonempty. Of course, $\supp S^0(\D^n)$ is nonempty too: Let $f=(f_1,...,f_n)\in \Ho(\D^n,\C^n)$,
then the evaluation $L(f)=f_1(z_0)$, $z_0\in \D^n\setminus\{0\},$ is an example for a continuous linear functional on $\Ho(\D^n,\C^n)$ such that $\Re L$ 
is non-constant on $S^0(\D^n).$

\begin{remark} Let $f\in \supp S^0(\D^n)$ be generated by the Herglotz vector field $G.$ Then, 
for a.e. $t\geq 0,$ $G(\cdot,t)\in \supp \M(\D^n).$ This is a consequence of Pontryain's maximum principle, see 
\cite[Theorem 1.5]{pont}. We have 
$$ \supp \M(\D) =  \left\{-z\sum_{k=1}^m \lambda_k \frac{e^{i\alpha_k}+z}{e^{i\alpha_k}-z}\,|\,
m\in\N, \alpha_k \in \R, \lambda_k\geq 0, \sum_{k=1}^m\lambda_k=1\right\}, $$
see \cite[Theorem 1]{hm1983}. By using the Herglotz representation for the class $\mathcal{P}$, one obtains
$$\ex \M(\D) = \left\{-z\frac{e^{i\alpha}+z}{e^{i\alpha}-z}\,|\, \alpha \in \R\right\}.$$
There are no such formulas for the higher dimensional case. However, Voda 
obtained that mappings of the form 
$h(z)=-(z_1 p_1(z_{j_1}), ..., z_np_n(z_{j_n}))$ are extreme points of $\M(\D^n)$ (see \cite[Prop. 2.2.1]{Voda}), 
where each  $p_k$ has the form $p_k(z)=\frac{e^{i\alpha_k}+z}{e^{i\alpha_k}-z}$ for some 
 $\alpha_k\in \R.$\\
He also notes (\cite[p. 55]{Voda}) that 
there must be extreme points of $\M(\D^n)$ not having this form. 
\end{remark}

\begin{remark}
Assume that a generator $M\in\M(\D^n)$ has the special form $$M(z) = -p(z) \cdot z.$$ Then $p:\D^n \to \C$ has to map $0$ to $1$ and $\Re(p(z))>0$ for all $z\in \D^n.$ The set of all those 
generators forms a convex and compact subset of $\M(\D^n).$ There is a Herglotz representation for $p$ via 
certain measures on $(\partial \D)^n,$ see \cite{MR0674279, MR873021}. \\ However, also in this case, it seems to be rather difficult to
 determine extreme points of this class for $n\geq 2$. In \cite{MR1017849}, it is shown that there exists an extreme 
 point whose corresponding measure on $(\partial \D)^n$ is absolutely continuous when $n\geq 2$, in contrast to the extreme points 
 for the case $n=1$, which all correspond to point measures on $\partial \D.$
\end{remark}

Extreme points as well as support points of the class $S^0(\D)$ map $\D$ onto $\C$ minus a slit (which has increasing modulus when one 
runs through the slit from its starting point to $\infty$), see \cite[\S9.4-\S9.5]{MR708494}. 
In particular, they are unbounded mappings. It would be interesting to find similar geometric properties of
extreme and support points of $S^0(\D^n)$ when $n\geq 2.$ In this section, we prove the following statements concerning 
support and extreme points of $S^0(\D^n).$

\begin{theorem}\label{sup} 
Let $f\in\supp S^0(\D^n)$ and let $\{f_t\}_{t\geq0}$ be a normalized Loewner chain with $f_0=f$ such that $\{e^{-t}f_t\}_{t\geq 0}$ is a normal family on $\D^n,$ then $e^{-t}f_t\in\supp S^0(\D^n)$ for all $t\geq 0.$ 
\end{theorem}

\begin{theorem}\label{extreme} 
Let $f\in \ex S^0(\D^n)$ and let $\{f_t\}_{t\geq0}$ be a normalized Loewner chain with $f_0=f$ such that $\{e^{-t}f_t\}_{t\geq 0}$ is a normal family on $\D^n,$ then $e^{-t}f_t\in \ex S^0(\D^n)$ for all $t\geq 0.$ 
\end{theorem}

 Our proof for Theorem \ref{sup} generalizes ideas from a proof for the case $n=1$, which is described in \cite{hallenbeck1984linear}; 
 see also \cite{schl2} for the case of the unit ball. Theorem \ref{extreme} is proved for the unit ball in 
\cite[Theorem 2.1]{MR2943779} and we can simply adopt this proof for the polydisc.\\

First, we note that, given an evolution family $\varphi_{s,t}$ associated to a Herglotz vector field and a mapping $G\in S^0(\D^n)$,
then $e^{t-s} G(\varphi_{s,t})$ is also in $S^0(\D^n),$ which is mentioned in the proof of Theorem 2.1 in \cite{MR2943779} for the unit ball case.

\begin{lemma}\label{s0}
Let $G\in S^0(\D^n)$ and $t\geq 0.$ Furthermore, let  $\{f_u\}_{u\geq0}$ be a normalized Loewner chain  such that $\{e^{-u}f_u\}_{u\geq0}$ is a normal family and let $\varphi_{s,t}$ be the associated evolution family.  Then $e^{t-s} G(\varphi_{s,t}) \in S^0(\D^n)$ for every $0\leq s \leq t.$
\end{lemma}
\begin{proof}
Let $\{G(\cdot,u)\}_{u\geq0}$ be a normalized Loewner chain with $G(\cdot,0)=G$ such that $\{e^{-u}G(\cdot,u)\}_{u\geq0}$ is a normal family and let  $F(z,u):\D^n\times [0,\infty)\to \C^n$ be the mapping $$F(z,u)=\begin{cases}
e^{t-s} G(\varphi_{s+u,t}(z)), &\quad 0\leq u \leq t-s,\\
e^{t-s} G(z, u+s-t), &\quad u>t-s.                                                                   \end{cases}
$$ Then $\{F(\cdot,u)\}_{u\geq0}$ is a normalized Loewner chain, $F(\cdot,0)=e^{t-s}G(\varphi_{s,t})$ and $\{e^{-u}F(\cdot,u)\}_{u\geq 0}$ is a normal family. Thus $e^{t-s}G(\varphi_{s,t})\in S^0(\D^n).$ 
\end{proof}

\begin{proof}[Proof of Theorem \ref{extreme}.]
 Suppose that $e^{-t}f_t\not\in \ex S^0(\D^n)$ for some $t>0.$ Then $e^{-t}f_t=sa+(1-s)b$ for some $a,b\in S^0(\D^n)$ with $a\not=b$ and $s\in(0,1).$ As $f=f_t\circ \varphi_{0,t},$ we have
$$ f = s \cdot (e^ta\circ \varphi_{0,t}) +  (1-s)\cdot (e^tb\circ \varphi_{0,t}).$$
The functions $e^ta\circ \varphi_{0,t}$ and $e^tb\circ \varphi_{0,t}$ belong to $S^0(\D^n)$ according to Lemma \ref{s0}. Thus, as $f\in \ex S^0(\D^n)$, they are identical and the identity theorem implies $a=b$, a contradiction.
\end{proof}

Choosing $G(z)=z$ in Lemma \ref{s0} shows that $e^{t-s}\varphi_{t-s}\in S^0(\D^n).$ 

\begin{lemma}\label{pol}
Let $\varphi_{s,t}$ be defined as in Lemma \ref{s0} and let $h= e^{t-s}\varphi_{s,t}\in S^0(\D^n).$ Furthermore, let  $P:\C^n\to\C^n$ be a polynomial with $P(0)=0,$ $DP(0)=0,$ then there exists $\delta>0$ such that $$h+\eps e^{t-s} P(e^{s-t}h) \in S^0(\D^n)\quad \text{for all} \quad \eps\in\C \quad \text{with} \quad |\eps|<\delta.$$
\end{lemma}
\begin{proof} Let $g_\eps(z)=z+\eps P(z).$
Obviously we have $g_{\eps}(0)=0, \; Dg_{\eps}(0)=I_n.$\\
Now $\det(Dg_{\eps}(z))\to 1$ for $\eps\to 0$ uniformly on $\overline{\D^n},$ so $g_{\eps}$ is locally biholomorphic for $\eps$ small enough. In this case, for every $z\in \overline{\D^n}$, we have:\\ $$[Dg_{\eps}(z)]^{-1}=[I_n+\eps DP(z)]^{-1}=I_n-\eps DP(z) + \eps^2DP(z)^2+... = I_n-\eps\underbrace{(DP(z)+...)}_{:=U(z)\in \C^{n\times n}}.$$ 
Write $[Dg_{\eps}(z)]^{-1}g_{\eps}(z) = z+\eps P(z)-\eps U(z)z-\eps^2 U(z)P(z)=$
$(I_n+\eps M(z))z,$ with a matrix-valued function $M(z).$\\
Now we show that $g_\eps\in S^*(\D^n)$ for $|\eps|$ small enough.\\
Let $g_j(z)$ be the $j$-th component of $-[Dg_{\eps}(z)]^{-1}g_{\eps}(z)$. For $\eps\to 0,$ the function $g_j(z)/z_j$ converges uniformly to $-1$ on the set $K:=\overline{\{z\in\D^n \with  \|z\|_{\infty}=|z_j|>0\}}.$ Thus there exists $\delta >0$ such that 
$$\Re \left(\frac{g_j(z)}{z_j}\right)<0\quad \text{for all} \quad z\in K, j=1,...,n\quad \text{and all} \quad \eps\in\C \quad \text{with} \; |\eps|<\delta.$$ 
Hence, $g_\eps \in S^*(\D^n)\subset S^0(\D^n)$ for all $\eps$ small enough by Theorem \ref{Juergen}.\\

From Lemma \ref{s0} it follows that $e^{t-s}g_\eps(\varphi_{s,t})=e^{t-s}g_\eps(e^{s-t}h)=h+\eps e^{t-s}P(e^{s-t}h)\in S^0(\D^n).$\end{proof}

The next statement shows that a special class of bounded mappings are not support points of $S^0(\D^n).$
 
\begin{proposition}\label{marina}
 Let $\varphi_{s,t}$ be defined as in Lemma \ref{s0} and let $h=e^{t-s} \varphi_{s,t} \in S^0(\D^n).$ Then $h$ is not a support point of $S^0(\D^n).$
\end{proposition}
\begin{proof}
Assume that $h$ is a support point of $S^0(\D^n),$ i.e. there is a continuous linear functional $L:\Ho(\D^n,\C^n)\to \C$ such that $\Re L$ is non-constant on $S^0(\D^n)$ and $$\Re L(h)=\max_{g\in S^0(\D^n)}\Re L(g).$$ Let $P$ be a polynomial with $P(0)=0$ and $DP(0)=0.$ Then $h+\eps e^{t-s} P(e^{s-t}h)\in S^0(\D^n)$ for all $\eps\in\C$ small enough by Lemma \ref{pol}.\\
We conclude $$\Re L(P(e^{s-t}h))=\Re L(P(\varphi_{s,t}))=0,$$ otherwise we could choose $\eps$ such that $\Re L(h+\eps e^{t-s}P(e^{s-t}h))>\Re L(h).$\\ 
Now $\varphi_{s,t}(\D^n)$ is a Runge domain by Theorem \ref{range} b). Hence we can write any analytic function $g$ defined in $\D^n$ with $g(0)=0$ and $Dg(0)=0$ as $g=\lim_{k\to \infty}P_k(\varphi_{s,t}),$ where every $P_k$ is a polynomial with $P_k(0)=0$ and $DP_k(0)=0$. The continuity of $L$ implies $\Re L(g)=0.$ Hence $\Re L$ is constant on $S(\D^n),$ a contradiction.
\end{proof}

\begin{proof}[Proof of Theorem \ref{sup}.]
 Let $L$ be a continuous linear functional on $\Ho(\D^n,\C^n)$ such that $\Re L$  is non-constant on $S^0(\D^n)$ with $$\Re L(f)=\max_{g\in S^0(\D^n)} \Re L(g).$$
Fix $t\geq 0,$ then $f(z)=f_t(\varphi_{0,t}(z))$ for all $z\in \D^n$. Define the continuous linear functional  $$J(g):=L(e^t\cdot g\circ \varphi_{0,t})  \quad \text{for}\quad g\in \Ho(\D^n,\C^n).$$
Now we have $$J(e^{-t} f_t)=L(f)\quad \text{and}\quad \Re J(g)\leq \Re J(e^{-t}f_t) \quad \text{for all}\quad g\in \Ho(\D^n,\C^n).$$
Furthermore, $\Re J$ is not constant on $S^0(\D^n)$: as $e^t\varphi_{0,t}$ is not a support point of $S^0(\D^n)$ by Proposition \ref{marina}, we have $\Re J(\text{id})=\Re L(e^t \varphi_{0,t})< \Re L(f)=\Re J(e^{-t}f_t).$
\end{proof}

\section{Coefficients of degree 2}\label{sec_4}

In this section we consider the coefficient functionals for coefficients of degree $2$. 
Let $(f_1,...,f_n)\in S^0(\D^n)$. By taking a permutation of the functions 
$f_1,...,f_n$ (and the variables $z_1,...,z_n$), we obtain again a mapping in 
$S^0(\D^n).$ Hence it is sufficient to consider the coefficients of $f_1$ only. We write 
$$ f_1(z) = z_1 + \sum_{|\alpha|\geq 2} A_\alpha z^\alpha. $$
Here we use multiindices $\alpha=(\alpha_1,...,\alpha_n)\in \N_0^n$ with $|\alpha|:=\alpha_1 + 
... + \alpha_n,$ $z^\alpha:=z_1^{\alpha_1}\cdot...\cdot z_n^{\alpha_n}.$\\
We are interested in the continuous linear functional $f\mapsto A_\alpha$ and the maximum of $\Re A_\alpha$ over $S^0(\D^n).$ 
First, we note that $$\max_{f\in S^0(\D^n)}\Re(A_\alpha) = \max_{f\in S^0(\D^n)} |A_\alpha|.$$ 
This can be seen by the following lemma which implies that we can always ``rotate'' functions from $S^0(\D^n)$ such that $A_\alpha\in (0,\infty).$

\begin{lemma}\label{rot}${}$ \begin{itemize}
	\item[a)] Let $h\in \M(\D^n)$ and  $j(z)=(e^{-i\alpha_1}h_1,...,e^{-i\alpha_1}h_n)(e^{i\alpha_1}z_1,...,e^{i\alpha_n}z_n)$ 
for some $\alpha_1,...,\alpha_n\in \R.$ Then $j\in \M(\D^n).$
\item[b)] Let $f\in S^0(\D^n)$ and  $g(z)=(e^{-i\alpha_1}f_1,...,e^{-i\alpha_1}f_n)(e^{i\alpha_1}z_1,...,e^{i\alpha_n}z_n)$ 
for some $\alpha_1,...,\alpha_n\in \R.$ Then $g\in S^0(\D^n).$
\end{itemize}
\end{lemma}
\begin{proof} a) follows directly from the definition of $\M(\D^n)$ and b) can be shown by using a).
\end{proof}

\begin{remark}
The following version of the Bieberbach conjecture for the class $S^0(\D^n)$ has been suggested in \cite{Gon99}:
\begin{equation}\label{Bie}\text{$\left\| \frac1{k!}D^kf(0)(w,w,...,w) \right\|_\infty \leq k$ for all $k\geq 2$ and $w\in \partial \D^n.$}\end{equation}
Obviously, it is sufficient to consider the component function $f_1$ only. For $w\in\partial\D^n$, let $f_w:\D\to\C, f_w(\lambda)=f_1(\lambda w).$ Then 
the conjecture above is equivalent to:
$$ \left|\frac1{k!}f_w^{(k)}(0)\right| \leq k  \quad \text{for all $k\geq 2$ and $w\in \partial \D^n$.}$$
We refer to \cite{LLX15} and the references therein for results concerning this estimate. 
The conjecture is known to be true for $n=2,$ see \cite[Theorem 3]{MR1049182}. In particular, by choosing $w$ to be a standard unit vector, we obtain
\begin{equation}\label{second} |A_{\alpha}|\leq 2  \end{equation}
for all $\alpha$ with $\alpha_j=2$ for some $j=1,...,n$ and $\alpha_k=0$ otherwise.\\
Of course, the estimate for $|D^2f_1(0)(w,w)|$ also implies estimates for the coefficients of the polynomial $D^2f_1(0)(w,w)$, thus for 
all $A_\alpha$ with $|\alpha|=2$.
\end{remark}

We will prove the following sharp estimates for $A_\alpha$ with $|\alpha|=2$. 

\begin{theorem}\label{bie2} Let $n\geq 2$ and $(f_1,...,f_n)\in S^0(\D^n)$, 
$f_1(z)=z_1 + \sum_{|\alpha|\geq 2} A_\alpha z^\alpha.$  Then the following statements hold:
\begin{itemize}
\item[a)] $$ |A_{\alpha}| \leq 2$$
for all $\alpha$ with $|\alpha|=2$ and $\alpha_1\not=0$.
This estimate is sharp for all such $\alpha$ due to the mappings 
$$F_1(z)=\left(\frac{z_1}{(1-z_1)^2},  z_2,...,z_n\right) \quad \text{for } \alpha=(2,0,...,0),$$
$$ F_2(z)=\left(z_1 (1 + z_2)^2, z_2,...,z_n\right), 
F_3(z)=\left(\frac{z_1(1 + z_2)}{1 - z_2}, \frac{z_2}{1 - z_2}, z_3,...,z_n\right) \quad \text{for } \alpha=(1,1,0,...,0).$$
\item[b)] $$ |A_{\alpha}| \leq 1$$
for all $\alpha$ with $|\alpha|=2$ and $\alpha_1=0$. This estimate is sharp for all such $\alpha$ due to the mappings 
$$F_4(z)=\left(z_1+z_2^2, z_2,...,z_n\right), 
F_5(z)=\left(\frac{z_1 - z_1 z_2 + z_2^2}{1 - z_2},\frac{z_2}{1 - z_2},z_3,...,z_n\right)\quad \text{for }  \alpha=(0,2,0,...,0),$$
$$F_6(z)=\left(z_1+z_2z_3, z_2,...,z_n\right),$$ 
$$F_7(z)=
\left(z_1+\frac{z_2 z_3 (\log(1 + z_2) - 
  \log(1 + z_3))}{z_2 - z_3}, \frac{z_2}{1 + z_2}, \frac{z_3}{1 + z_3}, z_4,...,z_n\right)
\quad \text{for } \alpha=(0,1,1,0...,0).$$

\end{itemize}

\end{theorem}

The examples $F_2,...,F_7$, which all belong to $S^*(\D^n)$ (see the proof of Theorem \ref{bie2}), yield the following corollary.

\begin{corollary} The functional $\Re A_{\alpha}$, with $|\alpha|=2$ and $\alpha_1\not=2,$ is maximized over 
$S^0(\D^n)$ by bounded as well as unbounded mappings. The bounded support points can be chosen to be restrictions of automorphisms of 
$\C^n.$
\end{corollary}

For $n=1$ and every bounded $f\in S(\D)$, we find a Herglotz vector field $H$ and 
a time $T>0$ such that the mapping $e^{-T}f:\D\to \D$ can be written as 
$e^{-T}f = \varphi_{0,T}$, where $\varphi_{0,t}$ solves \eqref{Loewner_ODE} for $H$; see Problem 3 in 
\cite[Section 6.1]{Pom:1975}.  
With Proposition \ref{marina} we obtain the following statement about the reachable 
set of equation \eqref{Loewner_ODE}.

\begin{corollary}
For $n\geq 2,$ there exist bounded mappings $f\in S^0(\D^n)$ which 
don't have the form $e^T\varphi_{0,T}$, where $T>0$ and $\varphi_{0,t}$ 
is a solution to \eqref{Loewner_ODE}. 
\end{corollary}

\begin{question}
Are there bounded mappings belonging to $\ex S^0(\D^n)$ for $n\geq 2$? 
\end{question}

\section{Proof of Theorem \ref{bie2}}

For the function $f_1(z)=z_1 + \sum_{|\alpha|\geq 2} A_\alpha z^\alpha$, the case $|\alpha|=2$ splits in 
essentially four cases, namely
$$\alpha = (2,0,...,0),\quad \alpha = (1,1,0,...,0), \quad 
\alpha = (0,2,0,...,0), \quad \alpha=(0,1,1,0,...,0).$$
All other cases can be reduced to one of these four by changing the order of some variables. Furthermore, 
the recursive structure of the Loewner equation shows that variables
$z_j$ with $\alpha_j=0$ don't effect our calculations for the coefficient $A_\alpha$ (see equation \eqref{recursive}). 
Thus we will restrict ourselves to the cases $n=2$ and $n=3$ respectively, i.e. we consider the cases 
$$\alpha = (2,0),\quad \alpha = (1,1), \quad 
\alpha = (0,2), \quad \alpha=(0,1,1).$$ 

First, we prove the following estimates with a technique noticed by Bracci in \cite{bsp} (``shearing process'').

\begin{proposition}\label{cat} Let $(h_1,h_2)\in \M(\D^2)$, 
$h_1(z)=-z_1 + \sum_{|\alpha|\geq 2} c_\alpha z^\alpha.$ 
\begin{itemize}

\item[a)] We have $h_1(z_1,0)\in \M(\D)$ and $|c_{(n,0)}|\leq 2$ for all $n\geq 2$. This estimate is sharp due to  
$$H_1(z)=\left(-z_1\frac{-1+z_1}{-1-z_1},-z_2\right)\in \M(\D^2).$$
\item[b)] We have $\left(-z_1(1-\sum_{\alpha_2\geq 1} c_{(1,\alpha_2)} z_2^{\alpha_2}), h_2\right) \in \M(\D^2)$ and 
$|c_{(1,n)}|\leq 2$ for all $n\geq 1$. This estimate is sharp due to 
$$H_2(z)=\left(-z_1\frac{-1+z_2}{-1-z_2},-z_2\right), H_3(z)=\left(-z_1\frac{-1+z_2}{-1-z_2},-z_2(1-z_2)\right)\in \M(\D^2).$$
\item[c)] 
We have $(-z_1+c_{(0,2)}z_2^2,h_2)\in \M(\D^2)$ and $ |c_{(0,2)}| \leq 1$. This estimate is sharp due to 
$$H_4(z)=\left(-z_1 + z_2^2,-z_2\right), H_5(z)=\left(-z_1+z_2^2, -z_2(1-z_2)\right) \in \M(\D^2).$$
\item[d)] Assume that $(h_1,h_2,h_3)\in \M(\D^3)$, $h_1(z)=-z_1 + \sum_{|\alpha|\geq 2} c_\alpha z^\alpha.$ 
Then $|c_{0,1,1}|\leq 1$. This estimate is sharp due to 
$$H_6(z)=\left(-z_1 + z_2z_3,-z_2,-z_3\right), H_7(z)=\left(-z_1+z_2z_3, -z_2(1+z_2), -z_3(1+z_3)\right) \in \M(\D^3).$$
\end{itemize}
\end{proposition}
\begin{proof}${}$
\begin{itemize}
 \item[a)] This is just the one-dimensional case, see Remark \ref{dim_one}.
 \item[b)] Let $z_1=xe^{i\theta}, z_2=ye^{i\varphi},$ with  $\theta, \varphi\in\R, x,y\in [0,1), x\geq y, x>0.$
 Then we have 
\begin{eqnarray*} && 0\geq  \Re(h_1(z)/z_1) = -1 + \Re\left(\sum_{|\alpha|\geq 2} c_\alpha z^\alpha/z_1\right) = 
-1 + \sum_{|\alpha|\geq 2}  x^{\alpha_1-1}y^{\alpha_2} \Re\left(c_\alpha e^{i\theta(\alpha_1-1)+i\varphi\alpha_2}\right).
\end{eqnarray*}
Hence, integration with respect to $\theta$ over $[0,2\pi]$ leads to 
\begin{equation*}\label{begger2} 0\geq 
-1 + \sum_{|\alpha|\geq 2, \alpha_1=1}  y^{\alpha_2} 
\Re\left(c_\alpha e^{i\varphi\alpha_2}\right) = 
-1 + \Re\left(\sum_{\alpha_2\geq 1}  c_{(1,\alpha_2)} z_2^{\alpha_2}\right), \quad \text{or}\end{equation*}
\begin{equation*}\label{begger3} 0\leq 
\Re\left(1-\sum_{\alpha_2\geq 1} c_{(1,\alpha_2)} z_2^{\alpha_2}\right).\end{equation*}
Hence, the function $z_2\mapsto 1-\sum_{\alpha_2\geq 1} c_{(1,\alpha_2)} z_2^{\alpha_2}$ belongs to the class 
$\mathcal{P}$ and \eqref{Chris} says
$$ |c_{(1,\alpha_2)}| \leq 2. $$
\item[c)] We can assume that $c_{(0,2)}\in\R.$ Otherwise, we apply a rotation from Lemma \ref{rot} a). 
Let $z_1=xe^{i\theta}, z_2=ye^{i\theta/2},$ for  $\theta\in\R, x,y\in [0,1), x\geq y, x>0.$ Then we have 
\begin{eqnarray*} && 0\geq  \Re(h_1(z)/z_1) = -1 + \Re\left(\sum_{|\alpha|\geq 2} c_\alpha z^\alpha/z_1\right) = 
-1 + \sum_{|\alpha|\geq 2}  x^{\alpha_1-1}y^{\alpha_2} \Re\left(c_\alpha e^{i\theta(\alpha_1-1+\alpha_2/2)}\right)\\
&=& -1 + c_{(0,2)} y^2/x + \sum_{|\alpha|\geq 2, \alpha\not=(0,2)}  x^{\alpha_1-1}y^{\alpha_2} 
\Re\left(c_\alpha e^{i\theta(\alpha_1-1+\alpha_2/2)}\right).
\end{eqnarray*}

The term $\alpha_1-1+\alpha_2/2$ is $\not=0$ for all $\alpha\not=(0,2)$ with 
$|\alpha|\geq 2.$ Hence, integration with respect to $\theta$ over $[0,4\pi]$ leads to 
\begin{equation}\label{begger} 0\geq  -1 + c_{(0,2)} y^2/x\end{equation}
for all $x,y \in (0,1)$ with $0< x \geq y.$ As 
$$ \Re((-z_1+c_{(0,2)}z_2^2)/z_1)\leq -1+c_{(0,2)} |z_2|^2/|z_1| $$
for all $(z_1,z_2)\in \D^2, z_1\not=0,$ we conclude that $(-z_1+c_{(0,2)}z_2^2,h_2)$ belongs to 
$\M(\D^2).$\\
Inequality \eqref{begger} is clearly satisfied for all $x,y \in (0,1)$ with $0< x \geq y$ if and only if 
$|c_{(0,2)}| \leq 1.$

\item[d)] Now we use a rotation from Lemma \ref{rot} a) to achieve that $c_{(0,1,1)}, ic_{(0,3,0)}\in\R.$ 
Let $z_1=xe^{i\varphi}, z_2=ye^{i\varphi/3}, z_3=we^{i2\varphi/3}$ for  $\varphi\in\R, x,y,w\in [0,1), x\geq y, x\geq w, x>0.$
 Then we have 
\begin{eqnarray*} && 0\geq  \Re(h_1(z)/z_1) = -1 + \Re\left(\sum_{|\alpha|\geq 2} c_\alpha z^\alpha/z_1\right)\\ &=& 
-1 + \sum_{|\alpha|\geq 2}  x^{\alpha_1-1}y^{\alpha_2}w^{\alpha_3} 
\Re\left(c_\alpha e^{i\varphi(\alpha_1-1+\alpha_2/3+2\alpha_3/3)}\right)\\
&=& -1 + c_{(0,1,1)}\frac{yw}{x} +\sum_{|\alpha|\geq 2, \alpha\not=(0,1,1)}  x^{\alpha_1-1}y^{\alpha_2}w^{\alpha_3} 
\Re\left(c_\alpha e^{i\varphi(\alpha_1-1+\alpha_2/3+2\alpha_3/3)}\right).
\end{eqnarray*}
The term $\alpha_1-1+\alpha_2/3+2\alpha_3/3$ in the last sum is $=0$ only for $\alpha=(0,3,0)$. Hence, integration with respect to $\theta$ over $[0,6\pi]$ leads to 
\begin{equation} 0\geq 
-1 + c_{(0,1,1)}\frac{yw}{x} + \frac{y^{3}}{x} \Re\left(c_{(0,3,0)}\right) = 
-1 + c_{(0,1,1)}\frac{yw}{x}.\end{equation}

Hence, $$ |c_{(0,1,1)}| \leq 1. $$

\end{itemize}
It is easy to verify that $H_1,...,H_7$ all belong to $\M(\D^n)$ by using the very definition of $\M(\D^n)$.
\end{proof}

\begin{proof}[Proof of Theorem \ref{bie2}] Let $f\in S^0(\D^n)$ with $f=\lim e^t \varphi_{0,t}$ 
for a corresponding evolution family $\{\varphi_{s,t}\}_{0\leq s\leq t}$ with associated Herglotz vector field $H$.\\ 
We now prove the coefficient estimate for $A_\alpha$ by comparing coefficients in the Loewner equation 
\eqref{Loewner_ODE} for $t\mapsto \varphi_{0,t}$ together with the coefficient estimates from Proposition \ref{cat}. 
As these steps are the same for each case, we only consider case c), i.e. $\alpha=(0,2).$\\
Let $\varphi_{0,t}=(w_{1,t}, w_{2,t})$ and write $w_{1,t}(z)=e^{-t}z_1 + \sum_{|\alpha|\geq 2}a_\alpha(t)z^{\alpha}$. Furthermore, 
we write $H(\cdot,t)=(h_{1,t},h_{2,t})$ with $h_{1,t}(z)=-z_1+\sum_{\alpha}c_\alpha(t) z^\alpha.$
 The Loewner equation yields (we use $\dot{y}$ for $\frac{\partial y}{\partial t}$)
 
 \begin{equation}\label{recursive}
 \dot{w}_{1,t} = h_{1,t}(w_{1,t}, w_{2,t}) = -w_{1,t} + c_{(0,2)}(t)w_{2,t}^2 + ... 
 \end{equation}

As $w_{2,t}(z)=e^{-t}z_2 + ...$, comparing the coefficients for $z_2^2$ gives
$$ \dot{a}_{(0,2)}(t) = - a_{(0,2)}(t) + c_{(0,2)}(t) e^{-2t}, \qquad a_{(0,2)}(0)=0, \quad \text{which implies}$$
$$ e^t a_{(0,2)}(t) = \int_0^t c_{(0,2)}(s) e^{-s}\, ds.$$
With Proposition \ref{cat} c) we obtain
$$|e^t a_{(0,2)}(t)| \leq \int_0^t |c_{(0,2)}(s)| e^{-s}\, ds \leq \int_0^t e^{-s}\, ds = 
1 -  e^{-t}.$$
Hence $|A_{(0,2)}|=\lim_{t\to\infty} |e^t a_{(0,2)}(t)| = 1.$

Finally we prove that the mappings $F_1,...,F_5$ belong to $S^0(\D^n).$ Let $H_j$, $j=1,...,7$, be the mappings 
from Proposition \ref{cat}. It is easy to verify that $-(DF_j)^{-1}F_j=H_j.$ Hence, by Theorem \ref{Juergen}, 
 $F_j\in S^*(\D^n)\subset S^0(\D^n).$

\end{proof}

\end{document}